\documentclass[english,10pt]{amsart}
\usepackage{amsmath,amsthm,amssymb}
\usepackage{tikz-cd}
\usepackage{hyperref}
\DeclareMathOperator{\Pic}{Pic}
\DeclareMathOperator{\sing}{sing}
\DeclareMathOperator{\Supp}{Supp}
\DeclareMathOperator{\codim}{codim}

\DeclareMathOperator{\RatCurves}{RatCurves}
\DeclareMathOperator{\Chow}{Chow}
\DeclareMathOperator{\Locus}{Locus}
\DeclareMathOperator{\h}{h}
\DeclareMathOperator{\an}{an}
\DeclareMathOperator{\im}{im}

\title{Del Pezzo foliations with log canonical singularities}
\author{Jo\~ao Paulo \textsc{FIGUEREDO}}
\address{\noindent IMPA, Estrada Dona Castorina 110, Rio de
  Janeiro, 22460-320, Brazil} 
\email{joaoplf@impa.br}

\begin{document}

\newtheorem{thm}{Theorem}[section]
\newtheorem{lem}[thm]{Lemma}
\newtheorem{prop}[thm]{Proposition}
\newtheorem{cor}[thm]{Corollary}
\theoremstyle{definition}
\newtheorem{defi}[thm]{Definition}
\theoremstyle{remark}
\newtheorem{rmk}[thm]{Remark}

\begin{abstract}
  In this paper, we classify del Pezzo foliations of rank at least 3 on projective manifolds and with log canonical singularities in the sense of McQuillan.
  
  \medskip
  
  
  \medskip
  
  \noindent {\bf Classification codes}: 14M22, 37F75.
\end{abstract}

\maketitle

\section{Introduction}

A foliation $\mathcal{F}$ on a normal complex variety $X$ is defined by a subsheaf $T_{\mathcal{F}} \subset T_X$ which is saturated in $T_X$ and closed under the Lie bracket. The anti-canonical class $-K_{\mathcal{F}}$ of $\mathcal{F}$ is the divisor class associated to the determinant $\det(T_{\mathcal{F}})$, and thus generalizes the anti-canonical class $-K_X$ of $X$. By analogy with the classical case, one tries to infer geometric information about $\mathcal{F}$ via numerical properties of $K_{\mathcal{F}}$. 

Conjecturally, the Minimal Model Program (MMP) says that each complex manifold is birational to a complex variety $X$ with mild singularities, such that either $K_X$ is nef (numerically effective), or $X$ is a fibration with Fano fibers $F$ (i.e. $-K_F$ is ample). In particular, the study of Fano varieties is important in the context of classification of algebraic varieties.

Therefore, if one expects a similar classification in the theory of foliations, then one is led to the concept of Fano foliations. They are the foliations $\mathcal{F}$ such that $-K_{\mathcal{F}}$ is ample. Fano foliations were introduced and studied by Araujo and Druel in \cite{fanofoliations} and \cite{fanofoliations2}. We remark that, although we do not have, at the moment, a conjecture similar to the classical MMP to foliations in all dimensions, there is a foliated MMP for dimension at most $3$ (see \cite{Spicer2017} and \cite{Cascini2018}). In this foliated MMP, one must restrict the singularities of the foliations involved to the class of F-dlt singularities, which is analogous to the dlt singularities in the case of varieties. The class of log canonical singularities in the sense of McQuillan is the largest class of singularities for which one can still make sense of the steps of the MMP (see definition \ref{sing}).

 The index $\iota_X$ of a Fano manifold is the largest integer dividing $-K_X$ in $\Pic(X)$. By a theorem of Kobayashi-Ochiai, we always have $\iota_X \leq \dim(X)+1$, and moreover equality holds if, and only if, $X$ is a projective space. The index $\iota_{\mathcal{F}}$ of a Fano foliation $\mathcal{F}$ is defined in an analogous way, and it turns out that there is a similar result for Fano foliations:

\begin{thm}[{{\cite[Theorem 1.1]{Araujo2008}}}]
	Let $\mathcal{F} \subsetneq T_X$ be a Fano foliation of rank $r$ on a complex projective manifold $X$. Then $\iota_{\mathcal{F}} \leq r$, and equality holds only if $X \cong \mathbb{P}^n$.\label{index}
\end{thm}

The study of foliations on $\mathbb{P}^n$ is classical. For example, a foliation $\mathcal{F}$ of codimension $1$ on $\mathbb{P}^n$ is given by a global homogeneous 1-form $\omega$ on $\mathbb{P}^n$ such that $\codim(\sing(\omega)) \geq 2$, $\omega \wedge d\omega =0$, and $\iota_R \omega = 0$, where $R = \sum_i x_i \frac{\partial}{\partial x_i}$ is the radial vector field in $\mathbb{C}^{n+1}$. Suppose $\omega$ has degree $\nu$. Then $K_{\mathcal{F}} = \mathcal{O}_{\mathbb{P}^n}(\nu-1-n)$. We see that if $\nu < n + 1$, then $\mathcal{F}$ is Fano and the index of $\mathcal{F}$ is $\iota_{\mathcal{F}} = n+1-\nu$. The degree of a foliation in $\mathbb{P}^n$ can be defined for arbitrary rank, and we still have an explicit formula for its index. In the case of maximal index (which corresponds, in codimension 1, to forms of degree 2) and arbitrary rank $r$, the foliation is given by a linear projection $\mathbb{P}^n \dashrightarrow \mathbb{P}^{n-r}$ (see \cite[Chapitre 3]{DC}).

The next case to consider is when $\iota_{\mathcal{F}} = r-1$. These are called del Pezzo foliations. Their study was initiated by Araujo and Druel also in papers \cite{fanofoliations} and \cite{fanofoliations2}. One important property of del Pezzo foliations they show is that, in most cases, they are algebraically integrable:

\begin{thm}[{{\cite[Theorem 1.1]{fanofoliations}}}]
	Let $\mathcal{F}$ be a del Pezzo foliation on a complex projective manifold $X \not\cong \mathbb{P}^n$. Then $\mathcal{F}$ is algebraically integrable, and its general leaves are rationally connected. \label{algebraicdelpezzo}
\end{thm}

In \cite{fanofoliations}, the authors also classify del Pezzo foliations under restrictions on the singularities. More precisely, they prove that if an algebraically integrable del Pezzo foliation $\mathcal{F}$ with rank $r \geq 3$ on a complex projective manifold $X$ has log canonical singularities along a general leaf and $T_{\mathcal{F}}$ is locally free along the closure of a general leaf, then $X$ is either a $\mathbb{P}^k$-bundle over $\mathbb{P}^m$, or $\rho(X) = 1$.

In this paper we remove the locally freeness assumption and extend this classification. This assumption is not natural in the classical theory of foliations, and moreover it is hard to check in practice. Furthermore, there are explicit examples of del Pezzo foliations that do not satisfy this property. For instance, consider $X=\mathbb{P}^N$, with $N \geq 3$ and $\mathcal{F}$ the foliation on $X$ given by a pencil of singular quadrics, generated by a double hyperplane and a quadric cone with vertex on this plane. Then one sees that $\mathcal{F}$ is a del Pezzo foliation which is log canonical along a general leaf. Here the general leaf $Q$ is a quadric cone with vertex $V$. If $T_{\mathcal{F}}$ is locally free along $Q$, then $T_Q$ and $(T_{\mathcal{F}})_{|Q}$ are two reflexive sheaves which are isomorphic in $Q\setminus \{V\}$, and therefore they are isomorphic if $\dim(Q) \geq 2$. But this would imply that $Q$ is regular. Thus we see that $T_{\mathcal{F}}$ is not locally free along $Q$.

 In our paper, we prove the following classification theorem:

\begin{thm}[Theorems \ref{thm1.0}, \ref{thm1.1} and \ref{picard1}]
	Let $X$ be a complex projective manifold of dimension $n$ and let $\mathcal{F}$ be a del Pezzo foliation of rank $r \geq 3$ on $X$ having log canonical singularities in the sense of McQuillan. Then 
\begin{itemize}
 \item Either $X$ is a $\mathbb{P}^m$-bundle over $\mathbb{P}^{k}$;
\item Or X is a projective space or a smooth quadric hypersurface.
\end{itemize} \label{main}
\end{thm}

The main new ingredient used in the proof of the above theorem is the classification of log leaves of algebraically integrable del Pezzo foliations, given firstly in \cite{fanofoliations2} for log canonical singularities along a general leaf, and secondly in \cite{icm}, where it is shown that one can remove the log canonicity hypothesis. We also remark that the classification is incomplete: it remains to consider the case of del Pezzo foliations of rank $2$. 

Finally, as done in \cite{fanofoliations} and \cite{fanofoliations2}, one can classify the possible del Pezzo foliations in the ambient spaces appearing in Theorem \ref{main}. By a classification of Loray, Pereira and Touzet (\cite[Theorem 6.2]{lpt}), we know all del Pezzo foliations in the projective spaces, and by a classification of Araujo and Druel (\cite[Theorem 9.6]{fanofoliations} and \cite[Proposition 3.18]{fanofoliations2}), we know all del Pezzo foliations in projective space bundles and in quadric hypersurfaces.

Throughout the paper, we work over the field of complex numbers $\mathbb{C}$.

\subsection*{Acknowledgments}
The author would like to thank Carolina Araujo for several ideas contained in this paper and also for reading the drafts and suggesting improvements; St\'ephane Druel for reading the second draft, for finding a gap and for suggesting corrections; and Jorge Vit\'orio Pereira for useful advice on foliations. During the writing of this paper, the author received financial support from CNPq (process number 140605/2017-7), and from CAPES/COFECUB (process number 88887.192325/2018-00 / Ma 932/19).

\section{General background}

In this section, we define and discuss some properties of foliations on normal varieties. More precisely, we consider del Pezzo foliations, and we state some of their main properties, as proved in papers \cite{fanofoliations}, \cite{fanofoliations2} and \cite{characterization}. We also discuss families of rational curves, and del Pezzo manifolds with Picard number 1.

\begin{defi}
	Let $X$ be a normal variety. A foliation $\mathcal{F}$ on $X$ is defined by a saturated coherent subsheaf $T_{\mathcal{F}}$ of $T_X$ that is closed under the Lie bracket.
\end{defi}

\begin{rmk}
	In the classical setting, a foliation $\mathcal{F}$ is given by a decomposition of the variety by a union of leaves (see remark \ref{frobleaves}), and to this foliation we associate a tangent sheaf $T_{\mathcal{F}}$, which satisfies the properties of the above definition. Notationally it allows us to distinguish the pullbacks $f^*(\mathcal{F})$ and $f^*(T_{\mathcal{F}})$, for $f\colon Y \rightarrow X$ a morphism. The pullback $f^*(\mathcal{F})$ is a foliation on $Y$ and in general $T_{f^*(\mathcal{F})} \neq f^*T_{\mathcal{F}}$.
\end{rmk}

\begin{defi}
 Let $\varphi\colon X \dashrightarrow Y$ be a dominant rational map between normal varieties, with connected fibers. Let $\mathcal{F}$ be a foliation on $Y$. Let $X^0 \subset X$ and $Y^0 \subset Y$ be smooth open subsets such that $\varphi_{|X^0}\colon X^0 \rightarrow Y^0$ is a smooth morphism. The pullback $\varphi^*\mathcal{F}$ of $\mathcal{F}$ under $\varphi$ is defined as the foliation on $X$ such that ${T_{\varphi^*\mathcal{F}}}_{|X^0} = d\varphi_{|X^0}^{-1}({T_{\mathcal{F}}}_{|Y^0})$.
\end{defi}

We now give the definition of the rank of a foliation and its singular set.

\begin{defi}
 Let $\mathcal{F}$ be a foliation on a normal variety $X$. The rank $r$ of $\mathcal{F}$ is defined as the rank of the coherent sheaf $T_{\mathcal{F}}$, and the codimension $\codim(\mathcal{F})$ as $\dim(X) - r$.
\end{defi}

\begin{defi}[{{\cite[Definition 2.4]{fanofoliations}}}]
	Let $\mathcal{F}$ be a foliation of rank $r$ on a normal variety $X$. The inclusion $T_{\mathcal{F}} \hookrightarrow T_X$ induces a map $\varphi\colon (\Omega_X^r\otimes \det(T_{\mathcal{F}}))^{**} \rightarrow \mathcal{O}_X$. We define the singular locus of $\mathcal{F}$, denoted by $\sing(\mathcal{F})$, as the subscheme of $X$ with ideal sheaf given by the image of $\varphi$.
\end{defi}

\begin{rmk}
	By definition, $\sing(\mathcal{F})$ is a closed subscheme of $X$. Since $T_{\mathcal{F}}$ is saturated in $T_X$, $\sing(\mathcal{F})$ has codimension at least $2$ in $X$.	
\end{rmk}

\begin{rmk}
 Let $\mathcal{F}$ be a foliation of rank $r$ on a normal variety $X$ of dimension $n$. Then $\mathcal{F}$ is given by an injective homomorphism of sheaves $T_{\mathcal{F}}\rightarrow T_X$. Define the normal bundle of $\mathcal{F}$ as $\mathcal{N}_{\mathcal{F}} := (T_X/T_{\mathcal{F}})^{**}$. There is a short exact sequence of sheaves
 	\[ \begin{tikzcd}[column sep=large, row sep=large]
		0 \arrow[r] & T_{\mathcal{F}} \arrow[r] & T_X \arrow[r] & T_X/T_{\mathcal{F}} \arrow[r] & 0,
	\end{tikzcd}\]
and we see that $\det(\mathcal{N}_{\mathcal{F}}) \cong \det(T_X)\otimes \det(T_{\mathcal{F}})^*$. Moreover, dualizing gives a homomorphism of sheaves $\mathcal{N}_{\mathcal{F}}^* \rightarrow (\Omega_X^{1})^{**}$. Taking wedge product, we get a non-zero section of $(\Omega_X^{q} \otimes \det(\mathcal{N}_{\mathcal{F}}))^{**}$, where $q = n-r$. In particular we see that if $\h^0(X, (\Omega_X^{q} \otimes \det(\mathcal{N}_{\mathcal{F}}))^{**}) = 0$, then $X$ does not admit a foliation of rank $r$ and normal bundle $\mathcal{N}_{\mathcal{F}}$. \label{normal}
\end{rmk}

The next theorem is classical and the name ``foliation'' is derived from it. It is used to define the leaves of a foliation.

\begin{thm}[Frobenius]
	Let $X$ be a manifold of dimension $n$ and $\mathcal{F}$ a foliation on $X$ of rank $r$. If $x \notin \sing(\mathcal{F})$, then there exists an analytic neighborhood $U$ of $x$ in $X$, with coordinates $(x_1,\dots,x_n)$, such that
	$$(T_{\mathcal{F}})_{|U} \cong \bigoplus_{i=1}^r \mathcal{O}^{\an}_{U} \frac{\partial}{\partial x_i}. $$
\end{thm}

\begin{defi}
    Let $X$ be a normal variety and $\mathcal{F}$ a foliation on $X$ of rank $r$. A leaf of $\mathcal{F}$ is the image of an injective morphism of analytic varieties $\varphi\colon F \rightarrow X\setminus (\sing(\mathcal{F}) \cup \sing(X))$ such that
    \begin{itemize}
        \item $F$ is connected and has dimension $r$;
        \item $D\varphi \colon T_F \rightarrow T_X$ is injective;
        \item $\im(D\varphi_p) = (T_{\mathcal{F}})_p$ for every $p \in F$.
    \end{itemize}
\end{defi}

\begin{rmk}
	Let $X$ be a normal variety and let $x \in X \setminus (\sing(X) \cup \sing(\mathcal{F}))$. Then, by Frobenius's theorem, there is an analytic neighborhood $U$ of $x$ with coordinates $(x_1,\dots,x_n)$, such that $(T_{\mathcal{F}})_{|U} = \ker(\pi)$, where $\pi\colon U \rightarrow V$ is the projection $\pi(x_1,\dots,x_n) = (x_{r+1},\dots,x_n)$. Thus, by taking an open cover $\{U_i\}$ of $X \setminus (\sing(X) \cup \sing(\mathcal{F}))$, and gluing the fibers of the $\pi_i\colon U_i \rightarrow V_i$ giving $\mathcal{F}$, we see that for each $x \in X \setminus (\sing(X) \cup \sing(\mathcal{F}))$, there is a leaf of $\mathcal{F}$ through $x$. \label{frobleaves}
\end{rmk}

\begin{defi}
	Let $X$ be a normal variety and $\mathcal{F}$ a foliation on $X$. Let $Y \subset X$ be an analytic subvariety such that $Y \not\subset \sing(X) \cup \sing(\mathcal{F})$. We say that $Y$ is tangent to $\mathcal{F}$ if the homomorphism $T_Y \rightarrow {T_X}_{|Y}$ factors through ${T_{\mathcal{F}}}_{|Y} \rightarrow {T_X}_{|Y}$; we say that $Y$ is invariant by $\mathcal{F}$, if ${T_{\mathcal{F}}}_{|Y} \rightarrow {T_X}_{|Y}$ factors through $T_Y \rightarrow {T_X}_{|Y}$. \label{tangentinvariant}
\end{defi}

One special class of foliations are those whose leaves are algebraic varieties.

\begin{defi}
	Let $X$ be a normal variety and $\mathcal{F}$ a foliation on $X$. We say that $\mathcal{F}$ is algebraically integrable if every leaf of $\mathcal{F}$ is algebraic (i.e. it is open in its Zariski closure).
\end{defi}

We can define the canonical divisor of a foliation in the same way that it is done in the classical case of varieties. We will see that several geometric properties of a foliation are translated into properties of its canonical divisor.

\begin{defi}
	Let $X$ be a normal variety and $\mathcal{F}$ a foliation on $X$. The canonical class of the foliation $\mathcal{F}$ is the linear equivalence class of Weil divisors $K_{\mathcal{F}}$ on $X$ such that $\mathcal{O}_X(-K_{\mathcal{F}}) \cong \det(T_{\mathcal{F}})$, and we call any divisor in this class a canonical divisor of $\mathcal{F}$.
\end{defi}

To study the birational geometry of foliations, one defines notions of singularities of $\mathcal{F}$ analogous to singularities of pairs used in modern study of birational geometry.

\begin{defi}[{{\cite[Definition I.1.5]{mcquillan}}}]
	Let $X$ be a normal variety and $\mathcal{F}$ a foliation on $X$ such that $K_{\mathcal{F}}$ is $\mathbb{Q}$-Cartier. Let $\pi\colon \tilde{X} \rightarrow X$ be a proper birational morphism, and denote by $\tilde{\mathcal{F}}$ the pullback of $\mathcal{F}$ along $\pi$. Write
	$$K_{\tilde{\mathcal{F}}} = \pi^*K_{\mathcal{F}} + \sum a(E_i,\mathcal{F}) E_i,$$
	where the $E_i$'s are the exceptional divisors of $\pi$. We say that $\mathcal{F}$ is terminal (resp. canonical, resp. log terminal, resp. log canonical) in the sense of McQuillan if $a(E_i,\mathcal{F}) > 0$ (resp. $\geq 0$, resp. $> -\epsilon(E_i)$, resp. $\geq -\epsilon(E_i)$), for all $\pi$ and all $E_i$, where $\epsilon(D) = 0$ if $D$ is invariant under $\tilde{\mathcal{F}}$, and $\epsilon(D) = 1$ otherwise. \label{sing}
\end{defi}

Concerning the case when $\mathcal{F}$ is algebraically integrable, we can relate the singularities of $\mathcal{F}$ with the singularities of the closure of a general leaf, made precise in the following definition and remark.

\begin{defi}[{{\cite[Definition 3.11]{article}}}]
	Let $X$ be a normal variety and $\mathcal{F}$ an algebraically integrable foliation on $X$ such that $K_{\mathcal{F}}$ is $\mathbb{Q}$-Cartier. Let $i\colon F \rightarrow X$ be the normalization of the closure of a general leaf of $\mathcal{F}$. Then there exists an effective $\mathbb{Q}$-divisor $\Delta$ on $F$ such that $K_F + \Delta \sim i^*K_{\mathcal{F}}$. The pair $(F,\Delta)$ is called a general log leaf of $\mathcal{F}$. \label{genlogleaf}
\end{defi}

\begin{rmk}[{{\cite[Proposition 3.11]{fanofoliations}}}]
	If $\mathcal{F}$ is algebraically integrable and has singularities of a given type as in definition \ref{sing}, then the general log leaf $(F,\Delta)$ also has singularities of the same type, seen as a pair.
\end{rmk}

The particularity of algebraically integrable Fano foliations with log canonical singularities is that there is a common point in the closure of the general leaf. More precisely, there is a common log canonical center of the log leaf, concept which we introduce in the following definition.

\begin{defi}[{{\cite[Definition 2.24]{kollar2008birational}}}]
 Let $(X,D)$ be a pair. A subvariety $W \subset X$ is called a log canonical center for $(X,D)$ if there exists a proper birational morphism from a normal variety $\mu\colon Y \rightarrow X$ and a prime divisor $E$ on $Y$ with discrepancy $a(E,X,D) \leq -1$ such that $\mu(E) = W$.
\end{defi}

\begin{prop}[{{\cite[Proposition 3.14]{fanofoliations2}}}]
	Let $X$ be a $\mathbb{Q}$-factorial projective variety and $\mathcal{F}$ an algebraically integrable Fano foliation on $X$. If the general log leaf $(F,\Delta)$ of $\mathcal{F}$ has log canonical singularities, then there is a closed irreducible subset $T \subset X$ such that there exists a log canonical center $S$ of $(F, \Delta)$ whose image in $X$ is $T$. \label{logcenter}
\end{prop}

Our proof of the classification of del Pezzo foliations mainly uses two tools: Fujita's classification of del Pezzo manifolds and the theory of families of rational curves on varieties. In the following theorem, we state Fujita's classification, considering only the cases of Picard number $1$, which are the cases we will be interested in our proof.

\begin{thm}[{{\cite{Fujita1980} and \cite{Fujita1981}}}]
	Let $X$ be a Fano manifold of dimension $n$ with $\iota_X = n - 1$ and $\rho(X) = 1$. Then $X$ is isomorphic to one of the following:
	\begin{enumerate}
		\item A cubic hypersurface in $\mathbb{P}^{n+1}$.
		\item An intersection of two quadric hypersurfaces in $\mathbb{P}^{n+2}$.
		\item A linear section of the Grassmannian $G(2,5) \subset \mathbb{P}^9$ under the Plücker embedding.
		\item A hypersurface of degree $4$ in the weighted projective space $\mathbb{P}(2,1,\dots,1)$.
		\item A hypersurface of degree $6$ in the weighted projective space $\mathbb{P}(3,2,1,\dots,1)$.
	\end{enumerate} \label{delpezzo}
\end{thm}

The following two lemmas show that the manifolds given in theorem \ref{delpezzo} do not admit del Pezzo foliations. More precisely, if $X$ is a Fano manifold of Picard number 1 with $\iota_X = n - 1$, then $-K_X = (n-1)A$, where $n = \dim(X)$ and $A$ is the ample generator of $\Pic(X)$. If $X$ admits a del Pezzo foliation $\mathcal{F}$, then, since $-K_{\mathcal{F}} = (r-1)A$, we have by remark \ref{normal} that $\mathcal{F}$ induces a non-zero element of $\h^0(X,\Omega_X^{n-r}(n-r))$ (notice that $\det(\mathcal{N}_{\mathcal{F}}) = (n-1)A - (r-1)A = (n-r)A$).

\begin{lem}[{{\cite[Lemma 5.17]{fanodist}}}]
	Let $X$ be a smooth complete intersection in a weighted projective space $\mathbb{P}(a_0,\dots,a_N)$ of dimension $n \geq 3$ defined by homogeneous polynomials $f_1,\dots,f_c$ of degrees $2\leq d_1\leq\dots\leq d_c$. Suppose that $X_i := \{f_1=\dots=f_i=0\}$ is smooth for every $i \in \{1,\dots,c\}$. Then $\h^0(X, \Omega_X^q(q)) = 0$ for each $1 \leq q \leq n-1$. \label{completeintersection}
\end{lem}

\begin{lem}
	Let $G = G(2,5) \subset \mathbb{P}^9$ be the Grassmannian of planes in $\mathbb{C}^5$. Denote by $G_i$ a general linear section of $G$ of codimension $i$. Then, for $i \in \{0,1,2\}$, $G_i$ does not admit a del Pezzo foliation of rank $r \geq 3$. \label{grassmann}
\end{lem}
\begin{proof}
	The result follows from \cite[Lemma 6.1]{fanodist} for codimension 1, from \cite[page 205, number (3)]{araujo2017codimension} for codimension 2 and from \cite[Lemma 0.1]{Peternell1995} for codimension 3.
\end{proof}

We now consider the concept of families of rational curves.

\begin{defi}[{{\cite[Definition IV.2.1]{kollar1999rational}}}]
	Let $X$ be a normal projective variety. By a family of rational curves $H$ on $X$ we mean an irreducible subvariety of the normalized scheme $\RatCurves^n(X)$ parametrizing rational curves on $X$. We denote by $\Locus(H)$ the locus of $X$ swept by the curves of $H$. We say that $H$ is unsplit if it is proper, and minimal if, for a general point $x \in \Locus(H)$, the closed subset $H_x$ of $H$ parametrizing curves through $x$ is proper. We say that $H$ is dominating if $\overline{\Locus(H)} = X$.
\end{defi}

One important property of families of rational curves is the existence of quotients:

\begin{defi}[{{\cite[Definition IV.3.2]{kollar1999rational}}}]
	Let $H$ be a family of rational curves on $X$. Let $\overline{H}$ denote the closure of $H$ in $\Chow(X)$. Two points $x,y \in X$ are said to be $H$-equivalent if they can be connected by a chain of $1$-cycles from $\overline{H}$. If a rational curve $C$ from $H$ has normalization $g\colon \mathbb{P}^1\rightarrow C$, then we denote the corresponding point in $H$ by $[C]$ or $[g]$. \label{family}
\end{defi}

\begin{thm}[\cite{Campana1992},\cite{Kollar1992b}]
	The above relation is an equivalence relation on $X$. Moreover there exists a proper surjective equidimensional morphism $\pi_0\colon X_0 \rightarrow Y_0$ from a dense open subset of $X$ onto a normal variety whose fibers are $H$-equivalence classes. We call this map the $H$-rationally connected quotient of $X$. \label{ratquotient}
\end{thm}

\begin{rmk}
Follow the notation of definition \ref{family} and theorem \ref{ratquotient}. Suppose $X$ is $\mathbb{Q}$-factorial. Let $\ell$ be a curve in $X$ such that $\ell \cap X_0 \neq \emptyset$. Suppose $\pi_0(\ell)$ is not a point. Then $\ell$ is not numerically proportional to the general curve of $H$. Indeed, let $D$ be an effective divisor in $Y_0$ which intersects $\pi_0(\ell)$ and does not contain $\pi_0(\ell)$. Let $D'$ be the closure of $\pi_0^{-1}(D)$ in $X$. Since $X$ is $\mathbb{Q}$-factorial, $D'$ is $\mathbb{Q}$-Cartier. For a general element $[\ell'] \in H$, $\pi_0$ contracts $\ell'$, and thus $D'\cdot \ell' = 0$, while $D'\cdot \ell > 0$, showing that $\ell$ and $\ell'$ are not numerically proportional. \label{numeric}
\end{rmk}

Finally, the following lemma related to families of rational curves will be used in the proof of the classification theorem. For the notion of free curve, see \cite[Definition II.3.1]{kollar1999rational}.

\begin{lem}[{{\cite[Lemma 2.6]{fanofoliations2}}}]
	Suppose that $X$ is a Fano manifold with $\rho(X) = 1$. Suppose that there is an $m$-dimensional family $V$ of rational curves of degree $d = -K_X\cdot V$ on $X$ such that:
	\begin{itemize}
		\item all curves from $V$ pass through some fixed point $x \in X$; and
		\item some curve from $V$ is free.
	\end{itemize}
	Then $\iota_X \geq \frac{m+2}{d}$. \label{boundforindex}
\end{lem}

\begin{rmk}
By abuse of notation, we call a curve $\ell \subset X$ a line if $\ell \cdot L = 1$ for some ample line bundle $L$ on $X$.
\end{rmk}

\section{Statement and proof of results}

In this section we prove our results, as stated in the introduction. By theorem \ref{algebraicdelpezzo}, if $X$ admits a del Pezzo foliation, then $X$ is uniruled. Thus, we can consider minimal dominating families of rational curves on $X$, and consequently, take the rationally connected quotients associated to these families. We consider two cases:

\begin{itemize}
	\item For some minimal dominating family $H$ with quotient $\pi_0\colon X_0 \rightarrow Y_0$, the foliation $\mathcal{F}$ satisfies ${T_{\mathcal{F}}}_{|X_0} \not\subset T_{X_0/Y_0}$;
	\item For every minimal dominating family $H$ with quotient $\pi_0\colon X_0 \rightarrow Y_0$, the foliation $\mathcal{F}$ satisfies ${T_{\mathcal{F}}}_{|X_0} \subset T_{X_0/Y_0}$.
\end{itemize}

Notice that the first condition means that for some $H$, the general leaf of $\mathcal{F}$ is not contained in the general fiber of $\pi_0$, while the second one means that, for every $H$, the general leaf of $\mathcal{F}$ is contained in the general fiber of $\pi_0$. When $r \geq 3$, we use the classification of leaves of del Pezzo foliations, stated below, to conclude that either $X$ is a $\mathbb{P}^m$-bundle over $\mathbb{P}^{n-m}$ in the first case, or that $\rho(X) = 1$ in the second case.

\begin{prop}[{{\cite[Theorem 2.15]{fanofoliations2}} and \cite[Corollary 2.13]{characterization}}]
 Let $\mathcal{F}$ be an algebraically integrable del Pezzo foliation of rank $r \geq 2$ on a smooth projective variety $X$, with general log leaf $(F,\Delta)$ having log canonical singularities. Let $L$ be an ample divisor on $X$ such that $-K_{\mathcal{F}} \sim (r-1)L$. Then $(F,\Delta,L_{|F})$ satisfies one of the following conditions.
 
 \begin{enumerate}
  \item $(F,\mathcal{O}_F(\Delta), \mathcal{O}_F(L_{|F})) \cong (\mathbb{P}^r,\mathcal{O}_{\mathbb{P}^r}(2),\mathcal{O}_{\mathbb{P}^r}(1))$.

\item $(F,\mathcal{O}_F(\Delta), \mathcal{O}_F(L_{|F})) \cong (Q^r, \mathcal{O}_{Q^r}(1), \mathcal{O}_{Q^r}(1))$, where $Q^r$ is a smooth quadric hypersurface in $\mathbb{P}^{r+1}$.
 
 \item $(F,\Delta)$ is a cone over $(Q^m,H)$, where $Q^m$ is a smooth quadric hypersurface in $\mathbb{P}^{m+1}$ for some $2\leq m < r$, $H \in |\mathcal{O}_{Q^m}(1)|$, and $L_{|F}$ is a hyperplane section under this embedding.
 
  \item $(F,\mathcal{O}_F(\Delta), \mathcal{O}_F(L_{|F})) \cong (\mathbb{P}^2,\mathcal{O}_{\mathbb{P}^2}(1),\mathcal{O}_{\mathbb{P}^2}(2))$.
  
  \item $(F, \mathcal{O}_F(L_{|F})) \cong (\mathbb{P}_{\mathbb{P}^1}(\mathcal{E}),\mathcal{O}_{\mathbb{P}(\mathcal{E})}(1))$, and one of the following holds:
  
  \begin{enumerate}
  \item $\mathcal{E} = \mathcal{O}_{\mathbb{P}^1}(1) \oplus \mathcal{O}_{\mathbb{P}^1}(d)$ for some $d \geq 2$, and $\Delta \sim_{\mathbb{Z}} \sigma + f$, where $\sigma$ is the minimal section and $f$ a fiber of $\mathbb{P}(\mathcal{E}) \rightarrow \mathbb{P}^1$.
  
  \item $\mathcal{E} = \mathcal{O}_{\mathbb{P}^1}(2) \oplus \mathcal{O}_{\mathbb{P}^1}(d)$ for some $d \geq 2$, and $\Delta$ is a minimal section.
  
  \item $\mathcal{E} = \mathcal{O}_{\mathbb{P}^1}(1) \oplus \mathcal{O}_{\mathbb{P}^1}(1) \oplus \mathcal{O}_{\mathbb{P}^1}(d)$ for some $d \geq 1$, and $\Delta = \mathbb{P}_{\mathbb{P}^1}(\mathcal{O}_{\mathbb{P}^1}(1) \oplus \mathcal{O}_{\mathbb{P}^1}(1))$.
  \end{enumerate}
  
  \item $(F,\Delta)$ is a cone over $(C_d,B)$, where $C_d$ is the rational normal curve of degree $d$ in $\mathbb{P}^d$ for some $d \geq 2$, $B \in |\mathcal{O}_{\mathbb{P}^1}(2)|$, and $L_{|F}$ is a hyperplane under this embedding.
  
  \item $(F,\Delta)$ is a cone over the pair (5a) above, and $L_{|F}$ is a hyperplane section of the cone.
 \end{enumerate}\label{leaves}
\end{prop}

In the proof of theorem \ref{algebraicdelpezzo}, the following lemma is used. It will also be used in the proof of our first theorem. For the reader's convenience, we sketch its proof.

\begin{lem}[{{\cite[page 100, proof of theorem 1.1]{fanofoliations}}}]
Let $X$ be a projective manifold of dimension $n \geq 3$, and suppose $X \not\cong \mathbb{P}^n$. Let $\mathcal{F}$ be a foliation on $X$ of rank $r \geq 2$, and suppose that there is an ample line bundle $L$ on $X$ such that $-K_\mathcal{F} \sim (r-1)L$. Let $H$ be a minimal dominating family of rational curves on $X$, with rationally connected quotient $\pi_0\colon X_0\rightarrow Y_0$, and take $[g] \in H$ a general member. Then one of the following holds:
\begin{enumerate}
\item $g^*T_{\mathcal{F}} \cong \mathcal{O}_{\mathbb{P}^1}(1)^{\oplus(r-1)}\oplus \mathcal{O}_{\mathbb{P}^1}$, and $H$ is unsplit, or
\item $g^*T_{\mathcal{F}} \cong \mathcal{O}_{\mathbb{P}^1}(2) \oplus \mathcal{O}_{\mathbb{P}^1}$, $r=2$, or
\item $g^*T_{\mathcal{F}} \cong \mathcal{O}_{\mathbb{P}^1}(2) \oplus \mathcal{O}_{\mathbb{P}^1}(1)^{\oplus(r-3)}\oplus \mathcal{O}_{\mathbb{P}^1}^{\oplus 2}$, $r \geq 3$, and $H$ is unsplit.
\end{enumerate} \label{decomposition}
\end{lem}
\begin{proof}[Idea of proof]
Since $T_{\mathcal{F}}$ is a reflexive sheaf, the locus $Z$ where $T_\mathcal{F}$ is not locally free has codimension at least $2$. Since $H$ is dominating, a general member $[g] \in H$ avoids $Z$ by \cite[Proposition II.3.7]{kollar1999rational}. Thus $g^*T_{\mathcal{F}} \cong \mathcal{O}_{\mathbb{P}^1}(a_1)\oplus \dots \oplus \mathcal{O}_{\mathbb{P}^1}(a_r)$, with $a_1+\dots+a_r = r-1$. Since $X \not\cong \mathbb{P}^n$, by \cite[Lemma 6.10]{fanofoliations}, $g^*T_{\mathcal{F}}$ is not ample. Moreover, by \cite[IV.2.9]{kollar1999rational}, $g^*T_X \cong \mathcal{O}_{\mathbb{P}^1}(2)\oplus \mathcal{O}_{\mathbb{P}^1}(1)^{\oplus d} \oplus \mathcal{O}_{\mathbb{P}^1}^{\oplus (n-d-1)}$, for some $d \geq 0$. This shows that $g^*T_{\mathcal{F}}$ is one of the cases (1), (2), or (3), or $g^*T_{\mathcal{F}} \cong \mathcal{O}_{\mathbb{P}^1}(2) \oplus \mathcal{O}_{\mathbb{P}^1}(1)^{\oplus (r-2)} \oplus \mathcal{O}_{\mathbb{P}^1}(-1)$.

If this last case happens, then $T_{X_0/Y_0} \subset {T_{\mathcal{F}}}_{|X_0}$ and thus $g^*T_{X_0/Y_0} \cong \mathcal{O}_{\mathbb{P}^1}(2) \oplus \mathcal{O}_{\mathbb{P}^1}(1)^{\oplus (r-2)}$. This implies that $\pi_0$ is a projective space bundle and $\mathcal{F}$ is the pullback of a foliation by curves on $Y_0$, which would imply that $g^*T_{\mathcal{F}} \cong \mathcal{O}_{\mathbb{P}^1}(2) \oplus \mathcal{O}_{\mathbb{P}^1}(1)^{\oplus(r-2)}\oplus \mathcal{O}_{\mathbb{P}^1}$, a contradiction.
\end{proof}
%
%

We can now give our first result, whose proof is based on \cite[Theorem 8.1]{fanofoliations}.

\begin{thm}
	Let $X \not\cong \mathbb{P}^n$ be a projective manifold of dimension $n$ and let $\mathcal{F}$ be a del Pezzo foliation with log canonical singularities in the sense of McQuillan. If for some minimal dominating family of rational curves $H$, with associated rationally connected quotient $\pi_0\colon X_0 \rightarrow Y_0$, we have that ${T_{\mathcal{F}}}_{|X_0} \not\subset T_{X_0/Y_0}$, then $X$ is isomorphic to a $\mathbb{P}^m$-bundle over $\mathbb{P}^{n-m}$ and $r \leq 3$. \label{thm1.0}
\end{thm}

\begin{proof}
	Our proof will consist of showing that there is a minimal dominating family $H'$ (possibly equal to $H$) of rational curves in $X$, with associated rationally connected quotient $\pi_0'\colon X_0' \rightarrow Y_0'$ satisfying ${T_{\mathcal{F}}}_{|X_0'} \not\subset T_{X_0'/Y_0'}$, and such that if $[g] \in H'$ is a general member, then $g^*T_{\mathcal{F}}$ is in case (1) of lemma \ref{decomposition}. These conditions will imply, by \cite[Proposition 7.13]{fanofoliations}, that $X$ is $\mathbb{P}^m$-bundle over $\mathbb{P}^{n-m}$, and $r \leq 3$.

	We begin by remarking that since $X \not\cong \mathbb{P}^n$, the foliation $\mathcal{F}$ is algebraically integrable by theorem \ref{algebraicdelpezzo}. In our proof, we will denote by $n\colon F \rightarrow X$ the normalization of the closure of a general leaf of $\mathcal{F}$.

	By lemma \ref{decomposition}, we know that for $[g] \in H$ general, there are three possibilities for $g^*T_{\mathcal{F}}$, which we label (1), (2) and (3) according to the same lemma.

	If case (1) happens, then by \cite[Proposition 7.13]{fanofoliations} $\pi_0$ makes $X$ a $\mathbb{P}^m$-bundle over $\mathbb{P}^{n-m}$, and $r \leq 3$.

	Suppose we are in cases (2) or (3). Then, we have $T_{\mathbb{P}^1} \subset g^*T_{\mathcal{F}}$, and it follows that the general member of $H$ is tangent to $\mathcal{F}$. Thus, after shrinking $Y_0$ if necessary, we may assume that $T_{X_0/Y_0} \subset {T_\mathcal{F}}_{|X_0}$ (see \cite[Lemma 6.9]{fanofoliations}). This implies that there is a foliation $\mathcal{G}$ on $Y_0$ such that $\mathcal{F} = \pi_0^* \mathcal{G}$, and thus $\sing(\mathcal{F})$ is supported on fibers of $\pi_0$ union $X \setminus X_0$. Now $H$ induces a minimal dominating family of rational curves on $F$, which we also denote by $H$. Notice that the hypothesis ${T_{\mathcal{F}}}_{|X_0} \not\subset T_{X_0/Y_0}$ implies that $F$ is not $H$-rationally connected. By the classification of leaves in Proposition \ref{leaves}, we must have $F$ a projective space bundle and $r \leq 3$ (cases (5a), (5b) and (5c)), since in all other possibilities, $F$ will be $H$-rationally connected for any $H$ a minimal dominating family.
	
	Denote by $\varphi$ the projective space bundle structure on $F$ given by propostion \ref{leaves}. Then $H$ on $F$ will be either the family of rational curves contained in fibers of $\varphi$, or a family of curves transverse to $\varphi$. In this second case, there is another $\mathbb{P}^1$-bundle structure on $F$, transverse to $\varphi$, which we denote by $\psi$. By \cite[Lemma 2.12]{characterization}, $\Supp(\Delta) = n^{-1}(\sing(\mathcal{F}))$. Since $\sing(\mathcal{F})$ is supported on fibers of $\pi_0$ union $X \setminus X_0$, we conclude that $\Delta$ is supported on a finite union of fibers of $\varphi$ or $\psi$. Analysing the cases (5a), (5b) and (5c) of proposition \ref{leaves}, this will be only possible if we are in case (5b), $d = 2$, and $F = \mathbb{P}^1 \times \mathbb{P}^1$. In this case $\psi$ is the second projection $F \rightarrow \mathbb{P}^1$ and $\Delta$ is a fiber of $\psi$. Denote by $f$ and $\sigma$ a fiber and a minimal section of $\psi$, respectively.

	Now, since $n^*K_\mathcal{F} \sim K_F + \Delta$, we have $-K_{\mathcal{F}}\cdot n_*(\sigma)= -K_F\cdot \sigma - f\cdot \sigma = 1$. Denote by $\ell' = n(\sigma)$ and let $H'$ be a dominating family of rational curves on $X$ containing $[\ell']$. Since $-K_{\mathcal{F}}\cdot \ell' = 1$, $H'$ is unsplit. We claim that the general member of $H'$ cannot be tangent to $\mathcal{F}$. Indeed, suppose it was. Then, the pullback of a general element of $H'$ to $F$ would be linearly equivalent to a section of $\psi$. Since, by \cite[Lemma 2.12]{characterization}, $n^{-1}(\sing(\mathcal{F})) \sim f$, this would imply that the general element of $H'$ intersects $\sing(\mathcal{F})$. However, $\codim(\sing(\mathcal{F})) \geq 2$, and by \cite[Proposition II.3.7]{kollar1999rational}, the general member of $H'$ avoids any closed set of codimension at least $2$. This is a contradiction. Thus, by lemma \ref{decomposition}, $(g')^*(T_{\mathcal{F}}) \cong \mathcal{O}_{\mathbb{P}^1}(1) \oplus \mathcal{O}_{\mathbb{P}^1}$, for $[g'] \in H'$ general. Consider $\pi'\colon X_1 \rightarrow Y_1$ the rationally connected quotient associated to $H'$. The condition ${T_\mathcal{F}}_{|X_0} \not\subset T_{X_0/Y_0}$ implies that $\pi_0$ does not contract the intersection of the general leaf of $\mathcal{F}$ with $X_0$. For this to happen, $\pi_0$ cannot contract $\ell' \cap X_0$ (we have $\ell' \cap X_0 \neq \emptyset$ because $\pi_0$ is defined on $n(f)$). Thus, by remark \ref{numeric}, $\ell'$ is not numerically proportional to the general element of $H$. Now, if $F'$ is a general fiber of $\pi'$, then by \cite[Proposition IV.3.13.3]{kollar1999rational}, $N_1(F')$ is generated by members of $H'$. Since the members of $H'$ are numerically proportional, it follows that the general member of $H$ cannot be contained in the general fiber of $\pi'$. Thus we have ${T_{\mathcal{F}}}_{|X_1} \not\subset T_{X_1/Y_1}$. By \cite[Proposition 7.13]{fanofoliations}, $\pi'$ makes $X$ a $\mathbb{P}^{n-m}$-bundle over $\mathbb{P}^m$ and $r \leq 3$
\end{proof}

Now we consider the case where ${T_\mathcal{F}}_{|X_0} \subset T_{X_0/Y_0}$, whose proof is also based in \cite{fanofoliations}.

\begin{thm}
Let $X \not\cong \mathbb{P}^n$ be a projective manifold of dimension $n$ and let $\mathcal{F}$ be a del Pezzo foliation of rank $r\geq 3$ and log canonical singularities in the sense of McQuillan. If for every minimal dominating family of rational curves $H$, with associated rationally quotient $\pi_0\colon X_0 \rightarrow Y_0$, we have ${T_{\mathcal{F}}}_{|X_0} \subset T_{X_0/Y_0}$, then $\rho(X) = 1$. \label{thm1.1}
\end{thm}

\begin{proof}
Let $n\colon F\rightarrow X$ be the normalization of the closure of a general leaf of $\mathcal{F}$; by assumption $\dim(F) \geq 3$. By proposition \ref{leaves}, we know all the possibilities for $F$; notice that cases (4), (5a) and (5b) do not occur. Since $\mathcal{O}_F(L_{|F}) = \mathcal{O}
_F(1)$ in all these cases, if $\ell$ is a line in $F$, then $n(\ell)$ is a line in $X$. Notice that in all cases, $F$ is covered by lines. Take the family $H$ on $X$ induced by lines on $F$, for general $F$. Then $L \cdot H = 1$, and in particular, $H$ is unsplit.

Since $\pi_0$ is proper, the closures of the general leaves of $\mathcal{F}$ are contained in the fibers of $\pi_0$. If $Y_0$ is not a point, then proposition \ref{logcenter} implies that the general fibers of $\pi_0$ have a common point, a contradiction. Therefore $Y_0$ is a point and, by \cite[Proposition 2.3]{Araujo2008}, $\rho(X) = 1$.
\end{proof}

Next we consider the case of manifolds with Picard number $1$ admitting del Pezzo foliations of rank $r \geq 3$. We can also restrict ourselves to Fano manifolds, by the following observation:

\begin{lem}
	Let $X$ be a projective manifold with $\rho(X) = 1$. If $X$ is uniruled, then $X$ is Fano.
\end{lem}

Consider then the index $\iota_X$ of $X$. Then, by \cite[Theorem 1.1]{kobayashi1973characterizations}, $\iota_X = \dim(X) + 1$ if and only if $X \cong \mathbb{P}^n$, and $\iota_X = \dim(X)$ if and only if $X \cong Q^n$, a quadric hypersurface. Thus in order to show that $X \cong \mathbb{P}^n$ or $X \cong Q^n$, we will show that $\iota_X \geq \dim(X)$. To do so, we will use lemma \ref{boundforindex}: we will find a family $V$ of lines on $X$, through a fixed point, of dimension $\dim(X) - 2$ or $\dim(X) - 1$. 

By theorem \ref{algebraicdelpezzo}, we may suppose that $\mathcal{F}$ is algebraically integrable. To find these families, we will consider the leaves of $\mathcal{F}$. By theorem \ref{algebraicdelpezzo}, the general leaves of $\mathcal{F}$ are rationally connected. By proposition \ref{logcenter}, the closures of these general leaves have a point in common. In particular, this shows that there exists a family of rational curves tangent to $\mathcal{F}$, passing through a fixed point. We do not know, a priori, the degree or the dimension of this family. However, the classification of leaves in Proposition \ref{leaves} will ensure the existence of families of rational curves of dimension $\dim(X) - 2$ or $\dim(X) - 1$. To estimate the dimension of such families, we will use the lemma below. 

\begin{lem}
Let $X$ be a projective manifold. Let $T \subsetneq X$ be an irreducible closed subset. Suppose that for every $x \in X$ general, there exists a closed subset $V(x) \subset T$ of fixed dimension $s$, such that for every general $y \in V(x)$, there exists a line in $X$ through $x$ and $y$. Then there exists a point $y \in T$ and a family $S$ of lines in $X$ passing through $y$, such that $\dim(S) \geq \dim(X) - \dim(T) + s - 1$, and such that the general member of $S$ is free. \label{dimfamilias}
\end{lem}
\begin{proof}
The assumptions of the lemma assure the existence of a dominating family $H_T$ of lines on $X$ such that for every $x \in X$ general, there exists a closed subset $V(x) \subset T$ of fixed dimension $s$, such that for every general $y \in V(x)$, there exists a line in $H_T$ through $x$ and $y$. Consider $\pi\colon \mathcal{U} \rightarrow H_T$ the universal family associated to $H_T$, with evaluation morphism $\eta \colon \mathcal{U} \rightarrow X$ (see \cite[Definition I.3.10]{kollar1999rational}). Then $\pi$ is a $\mathbb{P}^1$-bundle and in particular $\dim(H_T) = \dim(\mathcal{U}) - 1$. Now, for a general $x \in X$, $\pi(\eta^{-1}(x))$ consists of lines through $x$, and moreover by the description of $\eta$ (it does not contract fibers of $\pi$), $\dim(\eta^{-1}(x)) = \dim(\pi(\eta^{-1}(x)))$. Since for every general point in $V(x)$ there is a line through $x$ and $y$, we conclude that $\dim(\pi(\eta^{-1}(x))) \geq s$. Thus $\dim(\mathcal{U}) \geq \dim(X) + s$.

Now, by assumption, the preimage of $T$ in $\mathcal{U}$ dominates $H_T$, and since $H_T$ is dominating, there is a component of this preimage which is generically finite over $H_T$. Take then $W$ an irreducible component of $\eta^{-1}(T)$ such that $\pi_{|W}$ is generically finite and dominant. Let $T' = \overline{\eta(W)}$. Then $\dim(W) = \dim(H_T)$, and for $y \in T'$ general, $\dim(W) = \dim(T') + \dim(\eta^{-1}(y))$. Take $S = \pi(\eta^{-1}(y))$. Then $\dim(S) = \dim(W) - \dim(T') = \dim(H_T) - \dim(T') = \dim(\mathcal{U}) - 1 - \dim(T') \geq \dim(X) + s - 1 - \dim(T)$.
\end{proof}
	
We are now ready to prove our final result:

\begin{thm}
	Let $X$ be a projective manifold of dimension $n$ with $\rho(X) = 1$, admitting a del Pezzo foliation $\mathcal{F}$ of rank at least $3$ and log canonical singularities in the sense of McQuillan. Then either $X \cong \mathbb{P}^n$ or $X \cong Q^n$. \label{picard1}
\end{thm}

\begin{proof}
	The idea of the proof it to use the description of the leaves of $\mathcal{F}$ given in proposition \ref{leaves} (when $\mathcal{F}$ is algebraically integrable) to construct a family of lines satisfying the requirements of lemma \ref{dimfamilias}, in such a way that $\dim(T) - s \leq 2$. Then, by lemma \ref{boundforindex}, we will get $\iota_X \geq n-1$. If $\iota_X \geq n$, then $X \cong Q^n$ or $X \cong \mathbb{P}^n$ by \cite[Theorem 1.1]{kobayashi1973characterizations}. We will show that the case $\iota_X = n - 1$ is impossible.
	
	If $\mathcal{F}$ is not algebraically integrable, then $X \cong \mathbb{P}^n$ by theorem \ref{algebraicdelpezzo}. Suppose that $\mathcal{F}$ is algebraically integrable. Denote by $n\colon F \rightarrow X$ the normalization of a general leaf of $\mathcal{F}$. By proposition \ref{logcenter}, there is a log canonical center of $(F,\Delta)$ whose image $T$ in $X$ is fixed. Since the rank of the foliation is at least $3$, proposition \ref{leaves} implies that $(F, \Delta)$ is one of the cases (1), (2), (3), (5c), (6) or (7) given by it. 
	
	We now apply lemma \ref{dimfamilias} in each of these cases. Notice that in all the cases above, if $\ell \subset F$ is a line with respect to $n^*(L)$, then $n(\ell) \subset X$ is also a line with respect to $L$, where $L$ is an ample generator of $\Pic(X)$.
	
	In case (1), if $x \in X$ and $y \in T$ are general, then there is a line through $x$ and $y$. Indeed, it suffices to take $n(\ell)$, where $\ell$ is the line in $F$ through a point of $n^{-1}(x)$ and a point of $n^{-1}(y)$. Thus, by lemma \ref{dimfamilias}, for $y \in T$ general, there is a family $S$ of lines in $X$ through $y$ satisfying $\dim(S) \geq n - 1$. We conclude, by lemma \ref{boundforindex}, that $\iota_X = n+1$ and therefore $X \cong \mathbb{P}^n$. 

	Suppose now we are in cases (2), (3) or (5c). Then $F$ is either $Q^r$, a possibly singular quadric, or a $\mathbb{P}^2$-bundle over $\mathbb{P}^1$. If $p \in F$ is general, then the lines in $F$ through $p$ form a divisor in $F$. In particular, intersecting this divisor with $\Delta$, we see that there is a divisor in $\Delta$ whose points are connected to $p$ by a line. This means that for general $x \in X$, $\dim(T) - \dim(V(x)) \leq 1$, and thus by lemma \ref{dimfamilias}, there is a family $S$ of lines in $X$ through a $y \in T$ fixed, such that $\dim(S) \geq n$. We conclude, by lemma \ref{boundforindex}, that $\iota_X \geq n$, which implies that $X \cong \mathbb{P}^n$ or $X \cong Q^n$.
	
	In cases (6) and (7) the log leaf $(F, \Delta)$ is a cone over a smooth variety.
	
	In case (6), the common log canonical center may be the vertex $V$, or one of the rulings $\Delta_1$ or $\Delta_2$ giving $\Delta$, where by a ruling we mean the cone over a point in $C_d$. In the first case, for every $x \in X$ general, $V(x) = T$. In the second case, for general $x \in X$, $V(x)$ has codimension at most $1$ in $T$. Thus, by lemma \ref{dimfamilias}, there exists a family of lines in $X$ passing through a fixed $y \in T$, with $\dim(S) \geq n$. We conclude that $\iota_X \geq n$, which implies that $X \cong \mathbb{P}^n$ or $X \cong Q^n$.
	
	In case (7), the common log canonical center may be $\Delta_1$ or $\Delta_2$ (the cones over $\sigma$ and $f$), or the intersection of $\Delta_1$ and $\Delta_2$, which is a ruling of the leaf. In the first case, $\dim(T) - \dim(V(x)) \leq 2$ for general $x \in X$. In the second case, $\dim(T) - \dim(V(x)) \leq 1$ for general $x \in X$. Thus, by lemma \ref{dimfamilias}, there is a family $S$ of lines in $X$, such that $\dim(S) \geq n - 3$. We conclude that $\iota_X \geq n -1$, which implies that $X \cong \mathbb{P}^n$ or $X \cong Q^n$ or $\iota_X = n- 1$.
	
	Therefore, to finish the proof, we just need to show that $\iota_X = n-1$ is impossible. Theorem \ref{delpezzo} gives all the possibilities for $X$ in case $\iota_X = n - 1$. It is then enough to show that the manifolds listed in theorem \ref{delpezzo} cannot have a del Pezzo foliation. For cases (1), (2), (4) and (5) of that theorem, $X$ is a complete intersection in a weighted projective space $\mathbb{P}(a_0,\dots,a_N)$, and the result follows from remark \ref{normal} and from lemma \ref{completeintersection} (observe that in our case, if $A$ is an ample generator of $\Pic(X)$, then $-K_X \sim (n-1)A$ and $-K_{\mathcal{F}} \sim (r-1)A$, and therefore $\det(\mathcal{N}_{\mathcal{F}}) \sim (n-r)A$). In case number (3) of theorem \ref{delpezzo}, the result follows from lemma \ref{grassmann}. This concludes the proof.
\end{proof}

\bibliographystyle{alpha}
\bibliography{2nddraft}

\begin{thebibliography}{KMM92}

\bibitem[ACM18]{fanodist}
Carolina Araujo, Mauricio Corr{\^e}a, and Alex Massarenti.
\newblock Codimension one {F}ano distributions on {F}ano manifolds.
\newblock {\em Communications in Contemporary Mathematics}, 20(05):1750058,
  2018.

\bibitem[AD12]{article}
Carolina Araujo and St{\'e}phane Druel.
\newblock On codimension 1 del {P}ezzo foliations on varieties with mild
  singularities.
\newblock {\em Mathematische Annalen}, 360, 10 2012.

\bibitem[AD13]{fanofoliations}
Carolina Araujo and St{\'e}phane Druel.
\newblock On {F}ano foliations.
\newblock {\em Advances in Mathematics}, 238:70 -- 118, 2013.

\bibitem[AD16]{fanofoliations2}
Carolina Araujo and St{\'e}phane Druel.
\newblock On {F}ano foliations 2.
\newblock In Paolo Cascini, James McKernan, and Jorge~Vit{\'o}rio Pereira,
  editors, {\em Foliation Theory in Algebraic Geometry}, pages 1--20, Cham,
  2016. Springer International Publishing.

\bibitem[AD17a]{characterization}
Carolina Araujo and St{\'e}phane Druel.
\newblock Characterization of generic projective space bundles and algebraicity
  of foliations.
\newblock {\em arXiv preprint arXiv:1711.10174}, 2017.
\newblock To appear in Comment. Math. Helv.

\bibitem[AD17b]{araujo2017codimension}
Carolina Araujo and St{\'e}phane Druel.
\newblock Codimension 1 {M}ukai foliations on complex projective manifolds.
\newblock {\em Journal f{\"u}r die reine und angewandte Mathematik (Crelles
  Journal)}, 2017(727):191--246, 2017.

\bibitem[ADK08]{Araujo2008}
Carolina Araujo, St{\'e}phane Druel, and S{\'a}ndor~J. Kov{\'a}cs.
\newblock Cohomological characterizations of projective spaces and
  hyperquadrics.
\newblock {\em Inventiones mathematicae}, 174(2):233, Aug 2008.

\bibitem[Ara19]{icm}
Carolina Araujo.
\newblock Positivity and algebraic integrability of holomorphic foliations.
\newblock In {\em Proceedings of the International Congress of Mathematicians
  2018 (ICM 2018)}, volume~1, pages 545--562. WORLD SCIENTIFIC, 2019.

\bibitem[Cam92]{Campana1992}
Fr{\'e}d{\'e}ric Campana.
\newblock Connexit\'{e} rationnelle des vari\'{e}t\'{e}s de {F}ano.
\newblock {\em Annales Scientifiques de l'\'{E}cole Normale Sup\'{e}rieure.
  Quatri\`eme S\'{e}rie}, 25(5):539--545, 1992.

\bibitem[CS18]{Cascini2018}
Paolo Cascini and Calum Spicer.
\newblock {MMP} for co-rank one foliation on threefolds.
\newblock {\em arXiv:1808.02711 [math.AG]}, 2018.

\bibitem[DC06]{DC}
Julie D{\'e}serti and Dominique Cerveau.
\newblock Feuilletages et actions de groupes sur les espaces projectifs.
\newblock {\em Mém. Soc. Math. Fr. (N.S.)}, no. 103:vi+124 pp., 2006.

\bibitem[Fuj80]{Fujita1980}
Takao Fujita.
\newblock On the structure of polarized manifolds with total deficiency one.
  {I}.
\newblock {\em Journal of the Mathematical Society of Japan}, 32(4):709--725,
  1980.

\bibitem[Fuj81]{Fujita1981}
Takao Fujita.
\newblock On the structure of polarized manifolds with total deficiency one.
  {II}.
\newblock {\em Journal of the Mathematical Society of Japan}, 33(3):415--434,
  1981.

\bibitem[KM08]{kollar2008birational}
J{\'a}nos Koll{\'a}r and Shigefumi Mori.
\newblock {\em Birational Geometry of Algebraic Varieties}.
\newblock Cambridge Tracts in Mathematics. Cambridge University Press, 2008.

\bibitem[KMM92]{Kollar1992b}
J\'{a}nos Koll\'{a}r, Yoichi Miyaoka, and Shigefumi Mori.
\newblock Rationally connected varieties.
\newblock {\em Journal of Algebraic Geometry}, 1(3):429--448, 1992.

\bibitem[KO73]{kobayashi1973characterizations}
Shoshichi Kobayashi and Takushiro Ochiai.
\newblock Characterizations of complex projective spaces and hyperquadrics.
\newblock {\em Journal of Mathematics of Kyoto University}, 13(1):31--47, 1973.

\bibitem[Kol99]{kollar1999rational}
J{\'a}nos Koll{\'a}r.
\newblock {\em Rational Curves on Algebraic Varieties}.
\newblock Ergebnisse der Mathematik und ihrer Grenzgebiete. 3. Folge / A Series
  of Modern Surveys in Mathematics. Springer Berlin Heidelberg, 1999.

\bibitem[LPT13]{lpt}
Frank Loray, Jorge~Vit{\'o}rio Pereira, and Fr{\'e}d{\'e}ric Touzet.
\newblock Foliations with trivial canonical bundle on {F}ano 3-folds.
\newblock {\em Mathematische Nachrichten}, 286(8-9):921--940, 2013.

\bibitem[McQ08]{mcquillan}
Michael McQuillan.
\newblock Canonical models of foliations.
\newblock {\em Pure and Applied Mathematics Quarterly}, 4, 01 2008.

\bibitem[PW95]{Peternell1995}
Thomas Peternell and Jaros\l aw~A. Wi\'{s}niewski.
\newblock On stability of tangent bundles of {F}ano manifolds with {$b_2=1$}.
\newblock {\em Journal of Algebraic Geometry}, 4(2):363--384, 1995.

\bibitem[Spi20]{Spicer2017}
Calum Spicer.
\newblock Higher-dimensional foliated {M}ori theory.
\newblock {\em Compositio Mathematica}, 156(1):1--38, 2020.

\end{thebibliography}

\end{document}